\documentclass[1 [leqno,11pt]{amsart}
\usepackage{amssymb, amsmath}
\allowdisplaybreaks[4]
\def\inte#1{
\displaystyle\mathop{#1\kern0pt}^\circ }

\def\dj{\Delta_j}

\def\djp{\Delta_{j'}}

\def\orr{\omega^r}
\def\ot{\omega^\theta}
\def\orr{\omega^r}
\def\oz{\omega^z}
\def\ur{u^r}
\def\ut{u^\theta}
\def\br{b^r}
\def\bt{b^\theta}
\def\bz{b^z}
\def\uz{u^z}
\def\Jr{J^r}
\def\Jz{J^z}
\let\e=\varepsilon

\let\d=\partial
\let\pa=\partial
\let\wt=\widetilde

\def\ma{{\mathfrak a}}
\def\cB{{\mathcal B}}
\def\cC{{\mathcal C}}

\def\cF{{\mathcal F}}

\def\cP{{\mathcal P}}

\def\cS{{\mathcal S}}

\let\f=\frac
\let\D=\Delta

\renewcommand{\div}{{\rm div}\,}

\newcommand{\loc}{{\rm loc}\,}
\newcommand{\Lip}{{\rm Lip}\,}

\newcommand{\Supp}{{\rm Supp}\,}

\newcommand\ds{\displaystyle}

\newcommand{\Rmnum}[1]{\uppercase\expandafter{\romannumeral #1} }
 \numberwithin{equation}{section}

\let\s=\sigma

\def\dB{\dot{B}}

\def\virgp{\raise 2pt\hbox{,}}
\def\cdotpv{\raise 2pt\hbox{;}}

\def\eqdefa{\buildrel\hbox{\footnotesize def}\over =}

\def\C{\mathop{\mathbb C\kern 0pt}\nolimits}
\def\DD{\mathop{\mathbb D\kern 0pt}\nolimits}
\def\EE{\mathop{{\mathbb E \kern 0pt}}\nolimits}
\def\K{\mathop{\mathbb K\kern 0pt}\nolimits}
\def\N{\mathop{\mathbb N\kern 0pt}\nolimits}
\def\Q{\mathop{\mathbb Q\kern 0pt}\nolimits}
\def\R{\mathop{\mathbb R\kern 0pt}\nolimits}
\def\SS{\mathop{\mathbb S\kern 0pt}\nolimits}
\def\ZZ{\mathop{\mathbb Z\kern 0pt}\nolimits}
\def\TT{\mathop{\mathbb T\kern 0pt}\nolimits}
\def\PP{\mathop{\mathbb P\kern 0pt}\nolimits}

\newcommand{\Z}{{\ZZ}}

\def\dive{\mathop{\rm div}\nolimits}
\def\curl{\mathop{\rm curl}\nolimits}

\def\Supp{\mathop{\rm Supp}\nolimits\ }

\def\no{\noindent}
\def\na{\nabla}

\def\p3{\partial_3}

\def\u3{u^3}

\def\b3{b^3}

\def\om3{\omega^3}

\def\B0{B^0_{\infty,1}}

\newcommand{\beq}{\begin{equation}}
\newcommand{\eeq}{\end{equation}}
\newcommand{\ben}{\begin{eqnarray}}
\newcommand{\een}{\end{eqnarray}}
\newcommand{\beno}{\begin{eqnarray*}}
\newcommand{\eeno}{\end{eqnarray*}}
\newcommand{\andf}{\quad\hbox{and}\quad}
\newcommand{\with}{\quad\hbox{with}\quad}
\newtheorem{defi}{Definition}[section]
\newtheorem{thm}{Theorem}[section]
\newtheorem{lem}{Lemma}[section]
\newtheorem{rmk}{Remark}[section]

\begin{document}
\title[]
{On the global well-posedness of 3D axisymmetric MHD system with pure swirl magnetic field}

\author[Y. Liu]{Yanlin Liu}
\address[Y. Liu]{department of mathematical sciences, university of science and technology of china, Hefei 230026, CHINA,
and Academy of Mathematics $\&$ Systems Science, Chinese Academy of
Sciences, Beijing 100190, CHINA.} \email{liuyanlin3.14@126.com}

\date{\today}

\maketitle

\begin{abstract}
In this paper, we consider the axisymmetric MHD system with nearly critical initial data
having the special structure: $u_0=u_0^r e_r+\ut_0 e_\theta+u_0^z e_z,
~b_0=b_0^\theta e_\theta.$ We prove that, this system is global well-posed provided the scaling-invariant
norms $\|r\ut_0\|_{L^\infty},~\|r^{-1}\bt_0\|_{L^{\frac32}}$
are sufficiently small.
\end{abstract}



\section{Introduction}\label{sec1}
In this work, we investigate the global well-posedness
of the 3D axisymmetric MHD system.
In general, the 3D incompressible MHD system in the Euclidean coordinates
reads
\begin{equation}\label{MHD}
\left\{\begin{array}{ll}
\partial_{t}u + u\cdot \nabla u + \nabla P = \triangle u + b\cdot \nabla b,
\qquad(t,x)\in\R^+\times\R^3& \\
\partial_{t}b + u\cdot \nabla b = \triangle b + b\cdot \nabla u, & \\
 \rm{div}\ u =0, \quad \rm{div}\ b = 0, &\\
u|_{t=0}=u_0,\quad b|_{t=0}=b_0.&
\end{array}
\right.
\end{equation}
where $u$, $P$ denote the velocity and scalar pressure of the fluid
respectively, and $b$ denotes the magnetic field.
This system describes the time evolution of viscous
electrically-conducting fluids moving through a prevalent magnetic fields, such as plasmas, liquid metals, etc.

Note that when $b$ is identically zero, the system \eqref{MHD} reduces to the classical Navier-Stokes equations,
hence we can't expect to have a better theory on MHD than on the Navier-Stokes equations.
It is well-known that the global-wellposedness of 3D Navier-Stokes equations is still one of the most
challenging open problems in fluid mechanics, thus many efforts are made to study the solutions with some
special structures. The geometric structure axisymmetric is such an important case.
We call a vector field $v$ is axisymmetric if it can be written as
\begin{equation}\label{defaxi}
v(t,x)=v^r(t,r,z)e_r+v^\theta(t,r,z)e_\theta+v^z(t,r,z)e_z,
\end{equation}
where $(r,\theta,z)$ are the usual cylindrical coordinates in $\R^3$,
defined by $x=(r\cos\theta,r\sin\theta,z)$,~$r=\sqrt{x_1^2+x_2^2}$ for any $x\in\R^3$, and
$e_r=(\cos\theta,\sin\theta,0),e_\theta=(-\sin\theta,\cos\theta,0),e_z=(0,0,1)$.
$v^\theta$ is called the swirl component, and we say $v$ is axisymmetric without swirl if $v^\theta=0$.

For the axisymmetric without swirl solutions of Navier-Stokes equations,
Ladyzhenskaya \cite{La} and independently Ukhovskii and Yudovich
\cite{UY}  proved the
existence of weak solutions along with the uniqueness and regularities of such solutions,
\cite{LMNP} gived a
refined proof.
Abidi \cite{Abidi} gives global well-posedness in critical space $\dot{H}^{\frac12}$.

But for the case axisymmetric with non-trivial swirl, the global-wellposedness problem of Navier-Stokes equations is still open,
and seems as difficult as for the general case without any geometric structure.
The works for this case all need to put some smallness conditions on the initial data, see \cite{Chenhui,LZ,LeiZhang,Wei,ZZT2} for example.

For the general MHD system \eqref{MHD}, just as Navier-Stokes equations,
we also have local well-posedness result, and global well-posedness with small initial data,
see \cite{Paicu,Lions,Temam}.

Inspired by Lei \cite{Lei}, there he considered a family of special axisymmetric initial data
whose swirl components of the velocity field and magnetic vorticity field  are trivial, precisely
$$u(0,x)=u_0^r(r,z)e_r+u_0^z(r,z)e_z,
\ b(0,x)=b_0^\theta(r,z)e_\theta.$$
And he also assumed that the initial data are much regular satisfying $u_0,~b_0\in H^2,~r^{-1}\bt_0\in L^\infty$.
Then he can prove \eqref{MHD} is global well-posed, without any smallness assumptions.

In this paper, we consider the case where the swirl component of velocity is non-trivial:
\begin{equation}\label{initialstructure}
u(0,x)=u_0^r(r,z)e_r+u_0^\theta(r,z)e_\theta+u_0^z(r,z)e_z,
\ b(0,x)=b_0^\theta(r,z)e_\theta.
\end{equation}
As mentioned above, in this case, we need to handle some additional quadratic terms
caused by the swirl component $\ut$, and we can not expect global-wellposedness without any smallness assumptions.
Another difference is that, we only assume the initial data in the nearly critical spaces~(see Remark \ref{rmkcritical}),
not as regular as in \cite{Lei}, this brings some technical difficulties.

Before preceding, let us investigate the structure of the solutions and the equations.
It is classical that for axisymmetric initial data, the solutions remain axisymmetric
provided the solutions are regular enough. We claim that, if the solutions satisfy $u\in L^1(0,T;\Lip)$, then they
are not only axisymmetric, but also preserve the special form as the initial data:
\begin{equation}\label{solgeostru}
u(t,x)=u^r(t,r,z)e_r+u^\theta(t,r,z)e_\theta+u^z(t,r,z)e_z,
\ b(t,x)=b^\theta(t,r,z)e_\theta,\quad\forall t\in[0,T[.
\end{equation}
Indeed, we write the equations for $b$ in the cylindrical coordinates, to get
\begin{equation}\label{axi}
\left\{
\begin{array}{l}
\displaystyle \pa_t \br+(u^r\pa_r+u^z\pa_z) \br-(\Delta-\frac{1}{r^2})\br=
(\br\pa_r+\bz\pa_z)\ur,\\
\displaystyle \pa_t \bt+(u^r\pa_r+u^z\pa_z) \bt-(\Delta-\frac{1}{r^2})\bt=
(\br\pa_r+\bz\pa_z)\ut+\frac{\ur\bt}{r},\\
\displaystyle \pa_t b^z+(u^r\pa_r+u^z\pa_z) b^z-\Delta b^z=
(\br\pa_r+\bz\pa_z)\uz,\\
\displaystyle \pa_r (r b^r)+\pa_z (r b^z)=0.
\end{array}
\right.
\end{equation}
Applying $L^2$ estimate to $\br$ and $\bz$, and then integrating in time, we obtain
\begin{equation}\begin{split}
\|\br(t)\|_{L^2}^2+\|\bz(t)\|_{L^2}^2&\leq 2\int_0^t
\int_{\R^3}|\br(\br\pa_r+\bz\pa_z)\ur|+|\bz(\br\pa_r+\bz\pa_z)\uz|\,dx\,dt'\\
&\leq 4\int_0^t\bigl(\|\br\|_{L^2}^2+\|\bz\|_{L^2}^2\bigr)
\|\nabla u\|_{L^\infty}\,dt'.
\end{split}\end{equation}
Then Gronwall's inequality, combining with the initial data $\br_0=\bz_0=0$,
as well as $u\in L^1(0,T;\Lip)$,
guarantees that for any following time $t\in[0,T[$, we always have
$\br(t)=\bz(t)=0$, which is exactly the geometric structure \eqref{solgeostru}.
Thus we can reformulate \eqref{MHD} as
\begin{equation}\label{axi}
\left\{
\begin{array}{l}
\displaystyle \pa_t u^r+(u^r\pa_r+u^z\pa_z) u^r-(\Delta-\frac{1}{r^2})u^r
-\frac{(u^\theta)^2}{r}+\frac{(b^\theta)^2}{r}+\pa_r P=0,\\
\displaystyle \pa_t \ut+(u^r\pa_r+u^z\pa_z) \ut-(\Delta-\frac{1}{r^2})u^\theta+\frac{u^r u^\theta}{r}=0,\\
\displaystyle \pa_t u^z+(u^r\pa_r+u^z\pa_z) u^z-\Delta u^z+\pa_z P=0,\\
\displaystyle \pa_t \bt+(u^r\pa_r+u^z\pa_z) \bt-(\Delta-\frac{1}{r^2})\bt-\frac{\ur\bt}{r}=0,\\
\displaystyle \pa_r (r u^r)+\pa_z (r u^z)=0.
\end{array}
\right.
\end{equation}

For the axisymmetric velocity field $u$, we can write the vorticity
$\omega=\curl u$ as $\omega=\omega^r e_r+\omega^\theta e_\theta+\omega^z e_z,$
where
$\omega^r=-\pa_z u^\theta,~ \omega^\theta=\pa_z u^r-\pa_r u^z,~ \omega^z=\pa_r u^\theta+\frac{u^\theta}{r},$
satisfying
\begin{equation}\label{axiomega}
\left\{
\begin{array}{l}
\displaystyle \pa_t \omega^r+(u^r\pa_r+u^z\pa_z) \omega^r-(\Delta-\frac{1}{r^2})\omega^r-(\omega^r\pa_r+\omega^z\pa_z)u^r=0,\\
\displaystyle \pa_t \ot+(u^r\pa_r+u^z\pa_z) \ot-(\Delta-\frac{1}{r^2})\omega^\theta
-\frac{2u^\theta \pa_z u^\theta}{r}+\frac{2\bt \pa_z \bt}{r}-\frac{u^r \omega^\theta}{r}=0,\\
\displaystyle \pa_t \oz+(u^r\pa_r+u^z\pa_z)\oz-\Delta \omega^z-(\omega^r\pa_r+\omega^z\pa_z)u^z=0,\\
\displaystyle \omega|_{t=0} =\omega_0=\curl u_0.
\end{array}
\right.
\end{equation}

Denote $\widetilde{u}\eqdefa u^r e_r+u^z e_z,~\widetilde{\omega}\eqdefa \orr e_r+\oz e_z$. It is easy to check that
\begin{equation}\label{divcurl}
\dive \wt{u}=0 \andf \curl \wt{u} =\omega^\theta e_\theta,
\end{equation}
so that the Biot-Savart law shows that $u^r,~u^z$ can be uniquely determined by $\ot$.
Hence the System \eqref{axi} can be reformulated as
the equations for $\ut,~\bt$ and $\ot$.
Let us introduce another three variables which are of great importance in our work, namely
\begin{equation}\label{defetaV}
B\eqdefa \frac{\bt}{r},\quad
\eta\eqdefa\frac{\ot}{r},\quad V^\varepsilon\eqdefa \frac{\ut}{r^{1-\varepsilon}}
\quad \mbox{for any}\ \varepsilon\in]0,1[.
\end{equation}
Then we can use $B,~\eta$ and $V^\varepsilon$ to reformulate \eqref{axi} as follows:
\begin{equation}\label{etaV}
\left\{
\begin{array}{l}
\displaystyle \pa_t B+(u^r\pa_r+u^z\pa_z) B-(\Delta+\frac 2r\pa_r)B=0,\\
\displaystyle \pa_t \eta+(u^r\pa_r+u^z\pa_z) \eta-(\Delta+\frac 2r\pa_r)\eta
-\frac{2V \pa_z V}{r^{2\varepsilon}}+2B \pa_z B=0,\\
\displaystyle \pa_t V+(u^r\pa_r+u^z\pa_z) V+(2-\varepsilon)\frac{u^r V}{r}
-(\Delta+ 2(1-\varepsilon)\frac {\pa_r}{r})V
+(2\varepsilon-\varepsilon^2)\frac{V}{r^2}=0,\\
\displaystyle B|_{t=0} =B_0=\frac{\bt_0}{r},\quad
\eta|_{t=0} =\eta_0=\frac{\ot_0}{r},\quad V|_{t=0} =V_0=\frac{\ut_0}{r^{1-\varepsilon}},
\end{array}
\right.
\end{equation}
here and in all that follows, we always denote $V^\varepsilon$ as $V$,
if there is no ambiguity.

Our main result states as follows.
\begin{thm}\label{thmmain}
{\sl If there exists some $p\in\bigl]1,\f{63}{61}\,\bigr]$, and the corresponding $\e$ given by
\begin{equation}\label{thmmain1}
\e=\frac27-\frac{60(p-1)}{7(3-p)}\in\Bigl[\,\frac17,\frac27\,\Bigr[,
\end{equation}
such that the axisymmetric initial data $(u_0,\,b_0)$ with the special form \eqref{initialstructure} satisfying
\begin{equation}\label{thmmain2}
 u_0\in B^{-1}_{\infty,1},~\omega_0\in L^{\frac32},~ J_0\in L^{\frac32},
~ r\ut_0\in L^\infty\bigcap L^{\frac{\ma(p)}{p-1}},~ \eta_0\in L^p,
~ V^{\e}_0\in L^{\frac74},~B_0\in L^{\frac{3}2p},
\end{equation}
where $J_0\eqdefa \curl b_0,~\mathfrak{a}(p)\eqdefa \frac{(3-p)(23p-21)}{12(3p+1)}\in\bigl]\frac1{12},\frac{168}{1525}\bigr]$,
$B^{-1}_{\infty,1}$ is the homogeneous Besov space, see Definition \ref{defBesov}.
Furthermore, if
$\|r\ut_0\|_{L^\infty},~\|B_0\|_{L^{\frac32}}$
are small enough satisfying
\begin{equation}\begin{split}
&\|r\ut_0\|_{L^\infty}\leq c_0\min\biggl\{
(p-1)^{\f{20p}{7(3-p)}}(2M_0)^{-\f{4(p+2)}{7(3-p)}},
(p-1)^{8}(2M_0)^{-\f{2(p-1)}{p}}\|r\ut_0\|_{L^{\frac{\mathfrak{a}(p)}{p-1}}}^{-\f{23p-21}{2p}}
\biggr\},\\
&\|B_0\|_{L^{\frac32}}\leq c_0(p-1)^8(2M_0)^{-\frac{2(p-1)}{p}} \|B_0\|_{L^{\frac32 p}}^{-3},
\end{split}\end{equation}
where $c_0$ is some small universal constant,
and $M_0\eqdefa\|\eta_0\|_{L^p}^p+\|V_0\|_{L^{\frac74}}^{\frac74}+\|B_0\|_{L^{\frac{6p}{3+p}}}^{\frac{6p}{3+p}}$.
Then the MHD system \eqref{MHD} has a unique global solution in the space
$$ u\in
{L}^\infty\bigl([0,\infty[; B^{-1}_{\infty,1} \bigr)
\bigcap L^1_{\loc}\bigl([0,+\infty[;B^{1}_{\infty,1}\bigr),
\, \omega\in C\bigl([0,\infty[;L^{\frac32}\bigr),
\, J\in C\bigl([0,\infty[;L^{\frac32} \bigr),$$
$$\eta\in C\bigl([0,\infty[;L^{p}\bigr),\, V\in C\bigl([0,\infty[;L^{\frac74}\bigr),
\, B\in C\bigl([0,\infty[; L^{\frac{3}2p}\bigr),
\, r\ut \in C\bigl([0,\infty[;L^\infty\bigcap L^{\frac{\mathfrak{a}(p)}{p-1}}\bigr).
$$
Moreover, this solution admits the special form \eqref{solgeostru},
and for any $t\in [0,\infty[$, there holds
\begin{equation}\begin{split}\label{thmmain3}
\|\eta(t)\|_{L^p}^p+\|V(t)\|_{L^{\frac74}}^{\frac74}
+(p-1)\bigl\|\nabla|\eta|^{\frac p2}\bigr\|_{L^2_t L^2}^2
+\bigl\|\na|V|^{\frac 78}\bigr\|_{L^2_t L^2}^2
+\Bigl\|\frac{|V|^{\frac78}}{r}\Bigr\|_{L^2_t L^2}^2
\leq 2M_0,
\end{split}\end{equation}
\begin{equation}
\|\omega(t)\|_{L^{\frac32}}^{\frac32}
+ \bigl\|\nabla |\omega|^{\frac34}\bigr\|_{L^2_t L^2}^2
\leq C
\exp\Bigl\{C N_0 \exp\bigl\{C(2M_0)^{\frac{2}{2p-1}}
\cdot t^{\frac{3(p-1)}{2p-1}}\bigr\}\Bigr\},
\end{equation}
\begin{equation}
\|J(t)\|_{L^{\frac32}}^{\frac32}+
\bigl\|\nabla |J|^{\frac34}\bigr\|_{L^2_t L^2}^2
\leq C\exp\Bigl\{C N_0^{\frac53} \exp\bigl\{C(2M_0)^{\frac{2}{2p-1}}
\cdot t^{\frac{3(p-1)}{2p-1}}\bigr\}\Bigr\},
\end{equation}
\begin{equation}
\|u(t)\|_{B^{-1}_{\infty,1}}+\|u\|_{L^1_t B^{1}_{\infty,1}}
\leq \|u_0\|_{B^{-1}_{\infty,1}}+C\exp\Bigl\{C \bigl(N_0+N_0^{\frac53}\bigr) \exp\bigl\{C(2M_0)^{\frac{2}{2p-1}}
\cdot t^{\frac{3(p-1)}{2p-1}}\bigr\}\Bigr\},
\end{equation}
where $N_0\eqdefa \|\omega_0\|_{L^{\frac32}}^{\frac32}+\|J_0\|_{L^{\frac32}}^{3}
+\|r\ut_0\|_{L^\infty}^{\frac{3(39p-37)}{16(2p-1)}}
(2M_0)^{-\frac{3(p+2)}{4(2p-1)}}$.
}
\end{thm}

\begin{rmk}
{\sl A direct calculation, combining with the Biot-Savart law \eqref{Biot}, gives
\begin{equation}
\|B_0\|_{L^{\frac32}}\leq
\Bigl\|\frac{\pa_\theta}{r}(\bt_0 e_\theta)\Bigr\|_{L^{\frac32}}\leq
\|\nabla b_0\|_{L^{\frac32}}\leq C\|J_0\|_{L^{\frac32}}.
\end{equation}
Hence the initial data given by \eqref{thmmain2} in fact satisfy $B_0\in L^q$, $\forall q\in\bigl[\frac32,\frac32 p\bigr]$.
In particular, we have $B_0\in L^{\frac{6p}{3+p}}$, and thus $M_0$ is finite.
}\end{rmk}

\begin{rmk}
{\sl For the special case $\ut\equiv 0$, as considered in \cite{Lei},
the solutions are globally well-posed without any smallness assumptions.
So from the view of stability~(although the classical stability result in Besov spaces
can not be applied to the case here, since the norms here have weights with some powers of $r$), it seems more reasonable to make
smallness conditions only on $\ut_0$ in our theorem. I will explain in the following,
why I think the smallness assumptions on $\bt_0$ can not be dropped.

Let us see this system in a physics view. As we know, the way that the magnetic field
influence the velocity field is by the Lorentz force it induced.
Noting that the direction of the Lorentz force induced
by the pure swirl magnetic field $\bt e_\theta$ is orthogonal to $e_\theta$,
thus this force does not influence the energy of $r\ut$.
Indeed, precisely we have $\|r\ut(t)\|_{L^\ell}\leq\|r\ut_0\|_{L^\ell},\ \forall\ell\in]0,\infty]$,
see \eqref{estimateru}. Thus if initially we have $\ut_0=0$, then it must remain $0$ all the time.
That's the reason why the system is well-posedness in the case $\ut_0=0$,
as considered in \cite{Lei}, no matter how large $\bt_0$ is.
But this Lorentz force does influence the distribution of $\ut$.
Let us analyse this intuitively. Due to the equation for $\ur$ in \eqref{axi},
$\pa_t u^r=\cdots-r^{-1}|b^\theta|^2$, if $\bt$ is large, then $\ur$ may decrease and be negative
after a long time. And $\ur$ being negative means that the particles are moving toward the axis,
and this concentration to axis may lead to a finite time blow-up for $V=\frac{\ut}{r^{1-\varepsilon}}$.
One can also see this from the equation of $V$ \eqref{etaV}, if we want to control $V$,
then $(2-\varepsilon)\frac{u^r V}{r}$ can be a bad term when $\ur$ is negative,
and can be dangerous if in addition the absolute value of $\ur$ near the axis is too large.
From this point of view, it is not difficult to understand why we need to put some smallness
conditions on $\bt_0$.
We would also like to mention that, this observation indeed coincides with the fact that, the singularities of axisymmetric solutions
can only occur on the axis~(see the celebrated paper \cite{CKN} by CKN).
}\end{rmk}

\begin{rmk}\label{rmkcritical}
{\sl The MHD system has the same scaling property as Navier-Stokes equations.
Precisely, if $(u,b,P)$ are solutions of \eqref{MHD} on $[0,T]\times \R^3$,
then $(u,b,P)_\lambda$ defined by
\begin{equation} \label{scaling}
 (u,b,P)_\lambda(t,x) \eqdefa \bigl(
\lambda u (\lambda^2t, \lambda x), \lambda b (\lambda^2t, \lambda
x), \lambda^2 P(\lambda^2t,\lambda x)\bigr)
\end{equation}
are also solutions
on~$[0,\lambda^{-2} T]\times \R^3$
with initial data $\lambda u_0(\lambda x),~\lambda b_0(\lambda x)$.
And we call a norm is scaling-invariant, or critical,
if it does not change under this scaling transition \eqref{scaling}.

In the limiting case $p=1$, all the norms of initial data in Theorem \ref{thmmain}, namely
$$\|u_0\|_{ B^{-1}_{\infty,1}},~\|\omega_0\|_{ L^{\frac32}},~ \|J_0\|_{ L^{\frac32}},
~ \|r\ut_0\|_{ L^\infty},~ \|\eta_0\|_{L^1},
~ \bigl\|r^{-\frac57}\ut_0\bigr\|_{ L^{\frac74}},~\|B_0\|_{L^{\frac32}},$$
are scaling-invariant. And we consider the case $p>1$ but can be arbitrarily close to $1$,
that's what we mean by nearly critical.
}\end{rmk}
Of course, considering initial data not so regular brings some technical difficulties.
A direct difficulty is that, if the initial data has $H^2$ regularity, then we can
use $L^2$ framework, which is more convenient. But here we need to use $L^p$ framework, see Section \ref{Secglobal} for details.

Another difficulty is that,
in order to preserve the geometric structure \eqref{initialstructure} of the initial data,
we need to verify that the solutions we found indeed satisfy $u\in L^1(0,t;\Lip)$ for any $t>0$. This is not difficult
if the initial data are more regular. For example, \cite{Lei} considers the initial data $(u_0,b_0)\in H^2$, then a direct
energy estimate gives $u\in L^2(0,t;H^3)$, which implies $u\in L^1(0,t;\Lip)$.
Obviously, this estimate wastes a lot of regularity.
Noting that the norm $L^1(0,t;\Lip)$ is also critical, so for the nearly critical case here,
we can no longer waste so much regularity. Thus we can not use a direct energy estimate
to get the $L^2$ in time estimate, but use the integral formula of solutions and the estimates
for heat semi-group to get the $L^1$ in time estimate, and this is
the reason why we assume $u_0\in B^{-1}_{\infty,1}$, see Section \ref{SecL1} for details.
\smallbreak
We end up this section with some notations.
$C$ stands for some real positive constant which may be different in each occurrence.
Sometimes we use the notation $a\lesssim b$ for the inequality $a\leq Cb$ for some uniform constant C,
and $a\thicksim b$ means that both $a\lesssim b$ and $b\lesssim a$ hold. For a Banach space B, we shall use the shorthand $L^p_tB$ for $L^p(0,t;B)$.


\section{Basic facts on Littlewood-Paley theory}

For any $a\in\cS'$, let us recall the
dyadic decompositions of the Fourier variables as follows:
\begin{equation}\label{defparaproduct}
\Delta_j a=\cF^{-1}\bigl(\varphi(2^{-j}|\xi|)\widehat{a}\bigr),\quad
S_j a=\sum\limits_{j'\leq j-1}\djp a=\cF^{-1}\bigl(\chi(2^{-j}|\xi|)\widehat{a}\bigr),
\end{equation}
where $\cF a$ and $\widehat{a}$ denote the Fourier transform of $a$,
$\chi$ and $\varphi$ are smooth functions such that
\begin{align*}
&\ \Supp \varphi \subset \bigl\{\tau \in \R\,:  \, \frac34 \leq
|\tau| \leq \frac83 \bigr\}\andf  \forall
 \tau>0\,,\ \sum_{j\in\Z}\varphi(2^{-j}\tau)=1,\\
 &\Supp \chi \subset \bigl\{\tau \in \R\,: \, |\tau|  \leq
\frac43 \bigr\} \andf \forall\tau\geq 0 \,,\ \chi(\tau)+ \sum_{j\geq
0}\varphi(2^{-j}\tau)=1.
\end{align*}

Now we are in a position to define the
homogeneous Besov space $B^s_{p,r}$.

\begin{defi}\label{defBesov}
{\sl  Let $(p,r)$ in $[1,+\infty]^2$ and $s$ in $\R$. Let us consider $u$ in ${\mathcal
S}_h',$ which means that $u$ is in $\cS'$ and satisfies $\ds\lim_{j\to-\infty}\|S_ju\|_{L^\infty}=0$. We set
$$
\|u\|_{B^s_{p,r}}\eqdefa\big\|\big(2^{js}\|\Delta_j
u\|_{L^{p}}\big)_j\bigr\|_{\ell ^{r}(\ZZ)}.
$$
\begin{itemize}

\item
For $s<\frac{3}{p}$ (or $s=\frac{3}{p}$ if $r=1$), we define $
B^s_{p,r}\eqdefa \big\{u\in{\mathcal S}_h'\,:\, \|
u\|_{B^s_{p,r}}<\infty\big\}.$

\item
If $\frac{3}{p}+k\leq
s<\frac{3}{p}+k+1$ (or $s=\frac{3}{p}+k+1$ if $r=1$) for some $k\in\N$, then we
define $B^s_{p,r}$  as the subset of $u$
in ${\mathcal S}_h'$ such that $\partial^\beta u\in\dB^{s-k}_{p,r}$ whenever $|\beta|=k.$
\end{itemize}
We remark that in particular, $B^s_{2,2}$
coincide with the homogeneous Sobolev spaces $\dot{H}^s$.}
\end{defi}

The following Lemma \ref{Bern} is the well-known Bernstein inequality,
and Lemma \ref{GBern} can be seen as a generalization of the Bernstein inequality
and the Mihlin Multiplier Theorem.

\begin{lem}[Lemma 2.1 of \cite{BCD}]\label{Bern}
{\sl  Let $\cC$ be an annulus and $\cB$ a ball of $\R^d$. Then for
any nonnegative integer N, and $1\leq p\leq
q\leq\infty$, we have
$$\Supp \widehat{a}\subset\lambda\cB\Longrightarrow\|D^N a\|_{L^q}\eqdefa\sup_{|\alpha|=N}\|\pa^\alpha a\|_{L^q}
\lesssim\lambda^{N+3(\frac 1p-\frac 1q)}\|a\|_{L^p},$$
$$\Supp \widehat{a}\subset\lambda\cC\Longrightarrow \lambda^N\|a\|_{L^p}\lesssim\|D^N
a\|_{L^p}\lesssim\lambda^N\|a\|_{L^p}.$$}
\end{lem}

\begin{lem}[Lemma 2.2 of \cite{BCD}]\label{GBern}
{\sl  Let $\cC$ be an annulus, $m\in\R$ and $k=2[1+d/2]$.
If $\s$ is $k$-times differentiable on $\R^d\setminus\{0\}$,
and for any $\alpha\in\N^d$ with $|\alpha|\leq k$, there holds
\begin{equation}\label{conditionMihlin}
|\pa^\alpha\s(\xi)|\leq C_\alpha |\xi|^{m-|\alpha|},\quad \forall \xi\in\R^d\setminus\{0\}.
\end{equation}
Then for any $p\in[1,\infty]$, there exists a constant $C$ which depends only on $C_\alpha$,
such that
$$\Supp \widehat{a}\subset\lambda\cC\Longrightarrow \|\s(D)
a\|_{L^p}\leq C\lambda^m\|a\|_{L^p},\  \ \mbox{with}\ \
\s(D) a\eqdefa\cF^{-1}(\s\widehat{a}).$$}
\end{lem}

 Lemma \ref{heatflow}
studies the action of heat flow over spectrally supported functions.

\begin{lem}[Lemma 2.4 in \cite{BCD}]\label{heatflow}
{\sl Let $\cC$ be an annulus. Positive constants c and C exist such that for any $p\in[1,\infty]$ and $t,\lambda>0$,we have
$$\Supp \widehat{u}\subset\lambda\cC\Longrightarrow\|e^{t\Delta}u\|_{L^p}\leq Ce^{-ct\lambda^2}\|u\|_{L^p}.$$
}\end{lem}

Lemma \ref{productlaw} is the so-called tame estimates
for the product of two functions in Besov spaces.

\begin{lem}[Corollary 2.54 in \cite{BCD}]\label{productlaw}
{\sl If $(s,p,r)\in]0,\infty[\times[0,\infty]^2$ satisfies $s<\frac dp$, or $s=\frac dp$ and $r=1$,
then there exists a constant $C$ depending only on the dimension $d$, such that
$$\|uv\|_{B^s_{p,r}(\R^d)}\leq\frac{C^{s+1}}{s}
\bigl(\|u\|_{L^\infty(\R^d)}\|v\|_{B^s_{p,r}(\R^d)}+\|v\|_{L^\infty(\R^d)}\|u\|_{B^s_{p,r}(\R^d)}\bigr).$$
}\end{lem}

\section{The Global well-posedness of \eqref{etaV} in nearly critical spaces}\label{Secglobal}

Let us begin the proof of Theorem \ref{thmmain} with the global well-posedness of \eqref{etaV}.
As we are considering the nearly critical case, the $L^p$ framework, instead of the $L^2$ framework, is needed.
Thus the following lemma will play an important role in our global a prior estimates.
\begin{lem}[Proposition 2.1 of \cite{LZ}]\label{controlur}
{\sl Let  $p\in]1,3[$ and  $q\in\bigl]\frac{3p}{3-p},\infty\bigr].$ We assume that
$\eta\eqdefa\frac{\ot}{r}\in L^p(\R^3)\cap L^{3p}(\R^3)$. Then  we have
\begin{equation}\label{ureta}
\bigl\|\frac{u^r}{r}\bigr\|_{L^q}\lesssim\|\eta\|_{L^p}^{\lambda}\|\eta\|_{L^{3p}}^{1-\lambda}
\with  \lambda=\frac{p-1}{2}+\frac{3p}{2q}.
\end{equation}}
\end{lem}

Before preceding, let us prove the following elementary a priori estimates.
\begin{lem}\label{lemruB}
{\sl Let $(u^r,u^\theta, u^z,\bt)$ be a smooth enough solution of \eqref{axi} on $[0,T].$ Then for any
$\ell\in]1,\infty],~k\in]1,\infty[$, we have
\begin{equation}\label{estimateru}
\|r\ut(t)\|_{L^\ell}\leq\|r\ut_0\|_{L^\ell},\quad \forall \ t\in [0,T].
\end{equation}
\begin{equation}\label{estimateB}
\|B(t)\|_{L^{k}}^{k}
+\frac{4(k-1)}{k}\int_0^t \bigl\|\nabla|B(t')|^{\frac k2}\bigr\|_{L^2}^2\, dt'
\leq\|B_0\|_{L^{k}}^{k},
\quad \forall \ t\in [0,T].
\end{equation}
}
\end{lem}
\begin{proof}
The proof of \eqref{estimateru} can be found in Proposition 1 of \cite{CL02}.

As for \eqref{estimateB}, we get, by multiplying
the first equation of \eqref{etaV} by $|B|^{k-2}B$ and then integrating the
resulting equality over $\R^3$, that
\begin{equation}\label{energyB1}
\frac1{k}\frac{d}{dt}\|B(t)\|_{L^{k}}^{k}+\frac1{k}\int_{\R^3}(u^r\pa_r+u^z\pa_z)|B|^{k}\,dx
-\int_{\R^3}\D B  |B|^{k-2}B\,dx-\frac{2}{k}\int_{\R^3}\frac{\pa_r|B|^{k}}r\,dx=0.
\end{equation}
The divergence-free condition $\pa_r(ru^r)+\pa_z(ru^z)=0$ guarantees
$$\int_{\R^3}(u^r\pa_r+u^z\pa_z)|B|^k\,dx=0,$$
using this and integrating by parts, \eqref{energyB1} gives
\begin{equation}\label{energyB2}
\frac1{k}\frac{d}{dt}\|B(t)\|_{L^{k}}^{k}
+\frac{2}{k}\int_{-\infty}^{+\infty} |B|^{k}\big|_{r=0}\,dz
+\frac{4(k-1)}{k^2}\bigl\|\nabla|B|^{\frac k2}\bigr\|_{L^2}^2=0,
\end{equation}
then integrating in time gives the desired estimate \eqref{estimateB}.
\end{proof}

Now we are in a position to derive
the global well-posedness of \eqref{etaV}.

$\bullet$ \underline{The $L^{\frac74}$ estimate of $V$}

For the $\e$ given by \eqref{thmmain1},
applying $L^{\frac74}$ energy estimate for $V$ in \eqref{etaV},
we get
\begin{equation}\label{energyV1}
\frac47\frac{d}{dt}\|V(t)\|_{L^{\frac74}}^{\frac74}+
\frac{48}{49}\bigl\|\nabla |V|^{\frac78}\bigr\|_{L^2}^2+(2\e-\e^2)\int_{\R^3}\frac{|V|^{\frac74}}{r^2}\,dx
\leq (2-\e)\int_{\R^3}\bigl|\f{u^r}r\bigr|\cdot|V|^{\frac74}\,dx
\end{equation}

Taking $q_2=1+\frac{8}{7(2-\e)}=\frac{5p}{3(2p-1)}$,
and $q_1=\bigl(1-\frac{1}{q_2}\bigr)^{-1}=\frac{5p}{3-p}\in\bigl]\f{3p}{3-p},\infty\bigr[,$
so that we can use Lemma \ref{controlur}, Sobolev embedding theorem, and Young's inequality to get
\begin{align*}
\int_{\R^3}\bigl|\frac{u^r}{r}\bigr|\cdot|V|^{\frac74} dx\leq &\bigl\|\frac{u^r}{r}\bigr\|_{L^{q_1}}\bigl\||V|^{\frac74}\bigr\|_{L^{q_2}}\\
\lesssim &\|\eta\|_{L^p}^{\frac{p-1}{2}+\frac{3p}{2q_1}}\|\eta\|_{L^{3p}}^{\frac{3-p}{2}-\frac{3p}{2q_1}}
\Bigl(\int_{\R^3} \f{|V|^{\frac74}}{r^2} |r^{2-\e}V|^{\f{2}{2-\e}}\,dx\Bigr)^{\f1{q_2}}\\
\lesssim &\|\eta\|_{L^p}^{\frac{p-1}{2}+\frac{3p}{2q_1}}\bigl\|\na|\eta|^{\f{p}2}\bigr\|_{L^{2}}^{2\left(\frac{3-p}{2p}-\frac{3}{2q_1}\right)}
\Bigl(\int_{\R^3}\f{|V|^{\frac74}}{r^2}\,dx\Bigr)^{\f{1}{q_2}}
\|r\ut\|_{L^\infty}^{\f{2}{q_2(2-\e)}}\\
=&\|\eta\|_{L^p}^{\frac{p+2}{5}}\bigl\|\na|\eta|^{\f{p}2}\bigr\|_{L^{2}}^{2\bigl(\frac{3-p}{5p}\bigr)}
\Bigl(\int_{\R^3}\f{|V|^{\frac74}}{r^2}\,dx\Bigr)^{\f{3(2p-1)}{5p}}
\|r\ut\|_{L^\infty}^{\f{7(3-p)}{20p}}\\
\leq& \|r\ut\|_{L^\infty}^{\f{7(3-p)}{20p}}\|\eta\|_{L^p}^{\frac{p+2}{5}}
\Bigl(\bigl\|\na|\eta|^{\f{p}2}\bigr\|_{L^{2}}^2+\int_{\R^3}\f{|V|^{\frac74}}{r^2}\,dx\Bigr).
\end{align*}
Inserting this estimate into \eqref{energyV1}, we achieve
\begin{equation}\label{estimateV}\begin{split}
\frac47\frac{d}{dt}\|V(t)\|_{L^{\frac74}}^{\frac74}+
\frac{48}{49}\bigl\|\nabla |V|^{\frac78}& \bigr\|_{L^2}^2+\frac{48(2p-1)(33-31p)}{49(3-p)^2}\Bigl\|\frac{|V|^{\frac78}}{r}\Bigr\|_{L^2}^2\\
&\qquad\lesssim \|r\ut\|_{L^\infty}^{\f{7(3-p)}{20p}}\|\eta\|_{L^p}^{\frac{p+2}{5}}
\Bigl(\bigl\|\na|\eta|^{\f{p}2}\bigr\|_{L^{2}}^2+\Bigl\|\frac{|V|^{\frac78}}{r}\Bigr\|_{L^2}^2\Bigr).
\end{split}\end{equation}

\no$\bullet$ \underline{The $L^p$ estimate of $\eta$}

We get, by applying the $L^p$ energy estimate for $\eta$ in \eqref{etaV}, that
\begin{equation}\begin{split}\label{energyeta}
\frac1p\frac{d}{dt}\|\eta(t)\|_{L^p}^p
&+\frac{4(p-1)}{p^2}\bigl\|\nabla|\eta|^{\frac p2}\bigr\|_{L^2}^2
\leq \int_{\R^3}\bigl(\frac{2V \pa_z V}{r^{2\varepsilon}}-2B\pa_z B\bigr)|\eta|^{p-2}\eta\, dx\\
&\lesssim\bigl\||\eta|^{p-1}\bigr\|_{L^\frac{2p}{p-1}}\Bigl(
\bigl\|\pa_z|V|^{\frac 78}\bigr\|_{L^2}
\Bigl\|\frac{|V|^{\frac 98}}{r^{2\varepsilon}}\Bigr\|_{L^{2p}}
+\bigl\|\pa_z|B|^{\f{s}2}\bigr\|_{L^2}
\bigl\||B|^{2-\f{s}2}\bigr\|_{L^{2p}}\Bigr),
\end{split}
\end{equation}
where we take $s=\frac{6p}{3+p}$. It follows from Sobolev embedding Theorem that
\begin{equation}\label{Soboleveta}
\bigl\||\eta|^{p-1}\bigr\|_{L^\frac{2p}{p-1}}=\||\eta|^{\f{p}2}\|_{L^4}^{\f{2(p-1)}p}\lesssim \||\eta|^{\f{p}2}\|_{L^2}^{\f{p-1}{2p}}
\bigl\|\na|\eta|^{\f{p}2}\bigr\|_{L^2}^{\f{3(p-1)}{2p}}
=\|\eta\|_{L^p}^{\f{p-1}4}\bigl\|\na|\eta|^{\f{p}2}\bigr\|_{L^2}^{\f{3(p-1)}{2p}}.
\end{equation}
As a result, by the choice of $s=\frac{6p}{3+p}$,
we can use Young's inequality to obtain
\begin{equation}\label{RHSeta1}\begin{split}
\bigl\|&|\eta|^{p-1}\bigr\|_{L^\frac{2p}{p-1}}
\bigl\|\pa_z|B|^{\f{s}2}\bigr\|_{L^2}
\bigl\||B|^{2-\f{s}2}\bigr\|_{L^{2p}}\\
&\lesssim
\|\eta\|^{\frac{p-1}4}_{L^p}\bigl\|\nabla|\eta|^{\frac p2}\bigr\|_{L^2}^\frac{3(p-1)}{2p}
\bigl\|\pa_z|B|^{\f{s}2}\bigr\|_{L^2}
\bigl\||B|^{\f{s}2}\bigr\|_{L^2}^{\frac{3(p+1)}{2p}-1-\frac{2}{s}}
\bigl\|\nabla|B|^{\f{s}2}\bigr\|_{L^2}^{\frac{6}{s}-\frac{3(p+1)}{2p}}\\
&\lesssim
\|\eta\|^{\frac{p-1}4}_{L^p}\bigl\||B|^{\f{s}2}\bigr\|_{L^2}^{\frac{p+3}{6p}}
\bigl(\bigl\|\nabla|\eta|^{\frac p2}\bigr\|_{L^2}^2
+\bigl\|\nabla|B|^{\f{s}2}\bigr\|_{L^2}^{2}\bigr).
\end{split}\end{equation}

To handle the other term in \eqref{energyeta}, we split $\frac{|V|^{\frac98}}{r^{2\varepsilon}}$ as
\begin{equation*}
\frac{|V|^{\frac98}}{r^{2\varepsilon}}=\left|\frac{|V|^{\frac74}}{r^2}\right|^\alpha
\bigl|r^{2-\varepsilon}V\bigr|^{\f14+\frac{21(p-1)}{16p}} |V|^{\frac78\bigl(\frac{3-p}{2p}-2\alpha\bigr)},
\quad \mbox{with}\ \alpha=\bigl(\f14+\frac{21(p-1)}{16p}\bigr)\cdot \bigl(1-\f\e2\bigr)+\varepsilon.
\end{equation*}
Then we get, by applying H\"{o}lder's inequality
and Sobolev embedding Theorem, that
\begin{equation}\begin{split}\label{RHSeta2}
\Bigl\|\frac{|V|^{\frac98}}{r^{2\varepsilon}}\Bigr\|_{L^{2p}}
&\leq\Bigl\|\Bigl(\frac{|V|^{\frac74}}{r^2}\Bigr)^{\alpha}\Bigr\|_{L^{\frac{1}{\alpha}}}
\Bigl\||V|^{\frac78\bigl(\frac{3-p}{2p}-2\alpha\bigr)}\Bigr\|_{L^{6\cdot\bigl(\frac{3-p}{2p}-2\alpha\bigr)^{-1}}}
\Bigl\||r^{2-\varepsilon}V|^{\f14+\frac{21(p-1)}{16p}}\Bigr\|_{L^\beta}\\
&\leq \Bigl\|\frac{|V|^{\frac78}}{r}\Bigr\|^{2\alpha}_{L^2}
\bigl\|\nabla|V|^{\frac 78}\bigr\|_{L^2}^{\frac{3-p}{2p}-2\alpha}
\|r\ut\|_{L^{\beta\bigl(\f14+\frac{21(p-1)}{16p}\bigr)}}^{\f14+\frac{21(p-1)}{16p}},\\
\end{split}\end{equation}
where the index $\beta$ is given by (recalling $\e=\frac27-\frac{60(p-1)}{7(3-p)}$):
\begin{align*}
\frac1{\beta}=\frac{1}{2p}-\alpha-\frac16\bigl(\frac{3-p}{2p}-2\alpha\bigr)
=\frac{3(3p+1)}{4p(3-p)}\cdot(p-1).
\end{align*}

Using the estimates \eqref{Soboleveta},~\eqref{RHSeta2},
and a use of Young's inequality, gives rise to
\begin{equation}\label{RHSeta3}
\begin{split}
\bigl\|&|\eta|^{p-1}\bigr\|_{L^\frac{2p}{p-1}}
\bigl\|\pa_z|V|^{\frac 78}\bigr\|_{L^2}
\Bigl\|\frac{|V|^{\frac98}}{r^{2\varepsilon}}\Bigr\|_{L^{2p}}\\
&\lesssim
 \|\eta\|^{\frac{p-1}4}_{L^p}\|r\ut\|_{L^{\frac{(3-p)(25p-21)}{12(3p+1)}\cdot\frac{1}{p-1}}}^{\f{25p-21}{16p}}
\Bigl(\bigl\|\nabla|\eta|^{\frac p2}\bigr\|_{L^2}^2
+\bigl\|\na|V|^{\frac 78}\bigr\|_{L^2}^2
+\Bigl\|\frac{|V|^{\frac78}}{r}\Bigr\|_{L^2}^2\Bigr).
\end{split}
\end{equation}
Substituting the estimate \eqref{RHSeta1} and \eqref{RHSeta3} into the right hand side of \eqref{energyeta}, we achieve
\begin{equation}\begin{split}\label{estimateeta}
\frac1p\frac{d}{dt}\|\eta(t)\|_{L^p}^p
+\frac{4(p-1)}{p^2}&\bigl\|\nabla|\eta|^{\frac p2}\bigr\|_{L^2}^2
\lesssim
\|\eta\|^{\frac{p-1}4}_{L^p}
\Bigl(\|r\ut\|_{L^{\frac{\mathfrak{a}(p)}{p-1}\cdot\frac{25p-21}{23p-21}}}^{\f{25p-21}{16p}}+\|B\|_{L^s}^{\frac12}\Bigr)\\
&\times\Bigl(\bigl\|\nabla|\eta|^{\frac p2}\bigr\|_{L^2}^2
+\bigl\|\na|V|^{\frac 78}\bigr\|_{L^2}^2+\bigl\|\nabla|B|^{\f{s}2}\bigr\|_{L^2}^{2}
+\Bigl\|\frac{|V|^{\frac78}}{r}\Bigr\|_{L^2}^2\Bigr),
\end{split}\end{equation}
where $\mathfrak{a}(p)\eqdefa \frac{(3-p)(23p-21)}{12(3p+1)}\in\bigl]\frac1{12},\frac{168}{1525}\bigr]$.

\no$\bullet$ \underline{Continuity argument}

Denote $M(t)\eqdefa\|\eta(t)\|_{L^p}^p+\|V(t)\|_{L^{\frac74}}^{\frac74}+\|B(t)\|_{L^s}^s$ and $M(0)=M_0$.

Summing up \eqref{energyB2} with $k=s$,~\eqref{estimateV} and \eqref{estimateeta},
we get, by virtue of Lemma \ref{lemruB}, that
\begin{equation}\begin{split}\label{conti1}
\f{d}{dt}M(t)&
+2(p-1)\bigl\|\nabla|\eta|^{\frac p2}\bigr\|_{L^2}^2
+2\bigl\|\na|V|^{\frac 78}\bigr\|_{L^2}^2+2\bigl\|\nabla|B|^{\f{s}2}\bigr\|_{L^2}^{2}
+2\Bigl\|\frac{|V|^{\frac78}}{r}\Bigr\|_{L^2}^2\\
&\leq C
\left(\|r\ut_0\|_{L^\infty}^{\f{7(3-p)}{20p}}M(t)^{\frac{p+2}{5p}}
+\Bigl(\|r\ut_0\|_{L^{\frac{\mathfrak{a}(p)}{p-1}\cdot\frac{25p-21}{23p-21}}}^{\f{25p-21}{16p}}+\|B_0\|_{L^s}^{\frac12}\Bigr)
M(t)^{\frac{p-1}{4p}}\right)\\
&\qquad\qquad\qquad\qquad\times\Bigl(\bigl\|\nabla|\eta|^{\frac p2}\bigr\|_{L^2}^2
+\bigl\|\na|V|^{\frac 78}\bigr\|_{L^2}^2+\bigl\|\nabla|B|^{\f{s}2}\bigr\|_{L^2}^{2}
+\Bigl\|\frac{|V|^{\frac78}}{r}\Bigr\|_{L^2}^2\Bigr).
\end{split}\end{equation}
Next, we shall use a standard continuity argument.
Let $T^\ast>0$ be determined by
\begin{equation}\label{defT}
T^\ast\eqdefa\sup\Bigl\{T>0:\  \sup_{t\in [0,T[}M(t)
\leq 2M_0\ \Bigr\}.
\end{equation}
If $T^\ast<\infty$, then for any $t\leq T^\ast$,  we deduce from \eqref{conti1} that
\begin{equation}\begin{split}\label{conti2}
&\f{d}{dt}M(t)
+2(p-1)\bigl\|\nabla|\eta|^{\frac p2}\bigr\|_{L^2}^2
+2\bigl\|\na|V|^{\frac 78}\bigr\|_{L^2}^2+2\bigl\|\nabla|B|^{\f{s}2}\bigr\|_{L^2}^{2}
+2\Bigl\|\frac{|V|^{\frac78}}{r}\Bigr\|_{L^2}^2\\
&\leq C
\left(\|r\ut_0\|_{L^\infty}^{\f{7(3-p)}{20p}}\cdot(2M_0)^{\frac{p+2}{5p}}
+\Bigl(\|r\ut_0\|_{L^{\frac{\mathfrak{a}(p)}{p-1}\cdot\frac{25p-21}{23p-21}}}^{\f{25p-21}{16p}}+\|B_0\|_{L^s}^{\frac12}\Bigr)
\cdot(2M_0)^{\frac{p-1}{4p}}\right)\\
&\qquad\qquad\qquad\qquad\qquad\times\Bigl(\bigl\|\nabla|\eta|^{\frac p2}\bigr\|_{L^2}^2
+\bigl\|\na|V|^{\frac 78}\bigr\|_{L^2}^2+\bigl\|\nabla|B|^{\f{s}2}\bigr\|_{L^2}^{2}
+\Bigl\|\frac{|V|^{\frac78}}{r}\Bigr\|_{L^2}^2\Bigr).
\end{split}\end{equation}
Thus if there holds the smallness condition
\begin{equation*}
C \left(\|r\ut_0\|_{L^\infty}^{\f{7(3-p)}{20p}}\cdot(2M_0)^{\frac{p+2}{5p}}
+\Bigl(\|r\ut_0\|_{L^{\frac{\mathfrak{a}(p)}{p-1}\cdot\frac{25p-21}{23p-21}}}^{\f{25p-21}{16p}}+\|B_0\|_{L^s}^{\frac12}\Bigr)
\cdot(2M_0)^{\frac{p-1}{4p}}\right)\leq p-1,
\end{equation*}
which can be satisfied, with the help of the interpolation, by requiring that
\begin{equation}\begin{split}
&\|r\ut_0\|_{L^\infty}\leq c_0\min\biggl\{
(p-1)^{\f{20p}{7(3-p)}}(2M_0)^{-\f{4(p+2)}{7(3-p)}},
(p-1)^{8}(2M_0)^{-\f{2(p-1)}{p}}\|r\ut_0\|_{L^{\frac{\mathfrak{a}(p)}{p-1}}}^{-\f{23p-21}{2p}}
\biggr\},\\
&\|B_0\|_{L^{\frac32}}\leq c_0(p-1)^8(2M_0)^{-\frac{2(p-1)}{p}} \|B_0\|_{L^{\frac32 p}}^{-3},
\end{split}\end{equation}
for some small constant $c_0$.
Then \eqref{conti2} leads to, for any $t$ in $[0,T^\ast]$, that
$$
\f{d}{dt}M(t)
+(p-1)\bigl\|\nabla|\eta|^{\frac p2}\bigr\|_{L^2}^2
+\bigl\|\na|V|^{\frac 78}\bigr\|_{L^2}^2+\bigl\|\nabla|B|^{\f{s}2}\bigr\|_{L^2}^{2}
+\Bigl\|\frac{|V|^{\frac78}}{r}\Bigr\|_{L^2}^2\leq0.
$$
This  in particular gives rise to
$$
M(t)
\leq M_0\quad \mbox{for any}\quad t\leq T^\ast.
$$
This contradicts with the definition of $T^\ast$ given by \eqref{defT}. As a result, it comes out $T^\ast=\infty$, and there holds
\eqref{thmmain3} for any $t\in [0,\infty[$.

\section{Some global a priori estimates in critical spaces}

The purpose of this section is to derive the global a priori estimates for $\omega$ and $J$ in $L^{\frac32}$,
by using the global estimates of $B,~\eta,~V$ obtained already.
These estimates will be used to control the $L^1(0,t;\Lip)$ norm of $u$.
Recall the following well-known Biot-Savart law:
\begin{lem}\label{lemBiot}
{\sl There exists a constant $C$ depending only on the dimension $d$, such that for any $1<m<\infty$ and any
divergence-free vector field $u$, there holds
\begin{equation}\label{Biot}
\|\nabla u\|_{L^m(\R^d)}\leq C\frac{m^2}{m-1}\|\curl u\|_{L^m(\R^d)}.
\end{equation}
}\end{lem}

In the rest of this section, we shall first derive an $L^3$ estimate for $\bt$
~(noting that $\bt_0\in L^3$ since $\nabla b_0\in L^{\frac32}$),
which is needed in controlling the term $\frac{2\bt \pa_z \bt}{r}$
appearing in the equation of $\ot$, and then give the $L^{\frac32}$ estimate for $\omega$ and $J$.

$\bullet$ \underline{The $L^3$ estimate of $\bt$}

By applying the $L^3$ estimate for $\bt$ in \eqref{axi}, we get
\begin{align*}
\frac13\frac{d}{dt}\|\bt(t)\|_{L^{3}}^{3}+&
\frac89\bigl\|\nabla |\bt|^{\frac32}\bigr\|_{L^2}^2+\int_{\R^3}\frac{|\bt|^{3}}{r^2}\,dx
\leq\int_{\R^3}\bigl|\f{u^r}r\bigr|\cdot|\bt|^{3}\,dx\\
&\leq \bigl\|\f{u^r}r\bigr\|_{L^{3p}}\bigl\||\bt|^{\frac32}\bigr\|_{L^{\frac{6p}{3p-1}}}^2\\
&\leq C \|\eta\|_{L^p}^{\frac{p}2}\bigl\|\nabla|\eta|^{\frac p2}\bigr\|_{L^2}^{\frac 2p-1}
\bigl\||\bt|^{\frac32}\bigr\|_{L^2}^{2-\frac1p}
\bigl\|\nabla|\bt|^{\frac32}\bigr\|_{L^2}^{\frac1p}\\
&\leq \frac19\bigl\|\nabla |\bt|^{\frac32}\bigr\|_{L^2}^2
+C\|\eta\|_{L^p}^{\frac{p^2}{2p-1}}\|\bt\|_{L^3}^{3}
\bigl\|\nabla|\eta|^{\frac p2}\bigr\|_{L^2}^{\frac {2(2-p)}{2p-1}}.
\end{align*}
Absorbing the term $\frac19\bigl\|\nabla |\bt|^{\frac32}\bigr\|_{L^2}^2$, and then a use of Gronwall's inequality gives
\begin{align}\label{estimatebL3}
\|\bt(t)\|_{L^{3}}^{3}+
\int_0^t \bigl\|\nabla |\bt|^{\frac32}\bigr\|_{L^2}^2&+\Bigl\|\frac{|\bt|^{\frac32}}{r}\Bigr\|_{L^2}^2 dt'
\lesssim\|\bt_0\|_{L^{3}}^{3}\exp\Bigl\{C\int_0^t \|\eta\|_{L^p}^{\frac{p^2}{2p-1}}
\bigl\|\nabla|\eta|^{\frac p2}\bigr\|_{L^2}^{\frac {2(2-p)}{2p-1}} dt'\Bigr\}\notag\\
& \lesssim\|\bt_0\|_{L^{3}}^{3}\exp\Bigl\{C\|\eta\|_{L^\infty(0,t;L^p)}^{p\cdot\frac{p}{2p-1}}
\cdot t^{\frac{3(p-1)}{2p-1}}
\bigl(\int_0^t \bigl\|\nabla|\eta|^{\frac p2}\bigr\|_{L^2}^{2} dt'\bigr)^{\frac {2-p}{2p-1}}\Bigr\}\\
& \lesssim\|\bt_0\|_{L^{3}}^{3}\exp\Bigl\{C(2M_0)^{\frac{2}{2p-1}}
\cdot t^{\frac{3(p-1)}{2p-1}}\Bigr\},\notag
\end{align}
where we have used the estimate \eqref{thmmain3} in the last step.

$\bullet$ \underline{The $L^{\frac32}$ estimate of $\omega$}

Next, we derive the $L^{\frac32}$ estimate for $\omega$ in \eqref{axiomega}.
For $\ot$, we have
\begin{equation}\begin{split}\label{omega321}
\frac23\frac{d}{dt}\|\ot(t)\|_{L^{\frac32}}^{\frac32}+&
\frac89\bigl\|\nabla |\ot|^{\frac34}\bigr\|_{L^2}^2+\int_{\R^3}\frac{|\ot|^{\frac32}}{r^2}\,dx\\
&\leq \int_{\R^3}\bigl|\f{u^r}r\bigr|\cdot|\ot|^{\frac32}\,dx
+\int_{\R^3} \f{|2\ut\pa_z\ut-2\bt\pa_z\bt|}r \cdot|\ot|^{\frac12}\,dx.
\end{split}\end{equation}
The terms on the right hand side can be handled as follows:
\begin{align*}
\int_{\R^3}\bigl|\f{u^r}r\bigr|\cdot|\ot|^{\frac32}\,dx
&\leq \bigl\|\f{u^r}r\bigr\|_{L^{3p}}\bigl\||\ot|^{\frac34}\bigr\|_{L^{\frac{6p}{3p-1}}}^2\\
&\leq C \|\eta\|_{L^p}^{\frac{p}2}\bigl\|\nabla|\eta|^{\frac p2}\bigr\|_{L^2}^{\frac 2p-1}
\bigl\||\ot|^{\frac34}\bigr\|_{L^2}^{2-\frac1p}
\bigl\|\nabla|\ot|^{\frac34}\bigr\|_{L^2}^{\frac1p}\\
&\leq \frac19\bigl\|\nabla |\ot|^{\frac34}\bigr\|_{L^2}^2
+C\|\eta\|_{L^p}^{\frac{p^2}{2p-1}}\|\ot\|_{L^{\frac32}}^{\frac32}
\bigl\|\nabla|\eta|^{\frac p2}\bigr\|_{L^2}^{\frac {2(2-p)}{2p-1}}\,,
\end{align*}
and
\begin{align*}
&\int_{\R^3}\f{|\ut\pa_z\ut-\bt\pa_z\bt|}r \cdot|\ot|^{\frac12}\,dx
\leq \Bigl\|\f{\ut\pa_z\ut-\bt\pa_z\bt}r \Bigr\|_{L^{\frac98}}
\bigl\||\ot|^{\frac34}\|_{L^{6}}^{\frac23}\\
&\leq\biggl(\biggl\|\bigl(\pa_z|\bt|^{\frac32}\bigr)^{\frac23}\Bigl(\pa_z\Bigl|\frac{\bt}{r}\Bigr|^{\frac34}\Bigr)^{\frac13}
\Bigl|\frac{\bt}{r}\Bigr|^{\frac34} \biggr\|_{L^{\frac98}}
+\biggl\||r\ut|^{\frac{39p-37}{8(2p-1)}}
\Bigl|\frac{\ut}{r^{1-\e}}\Bigr|^{\frac78\cdot\frac{4-3p}{2p-1}}
\pa_z\Bigl|\frac{\ut}{r^{1-\e}}\Bigr|^{\frac78} \biggr\|_{L^{\frac98}} \biggr)
\bigl\|\nabla|\ot|^{\frac34}\|_{L^2}^{\frac23}\\
&\leq  \Bigl(\bigl\|\pa_z|\bt|^{\frac32}\bigr\|_{L^2}^{\frac23}
\bigl\|\pa_z|B|^{\frac34}\bigr\|_{L^2}^{\frac13}
\bigl\||B|^{\frac34}\bigr\|_{L^{\frac{18}7}}
+\|r\ut\|_{L^\infty}^{\frac{39p-37}{8(2p-1)}}
\bigl\|\pa_z|V|^{\frac78}\bigr\|_{L^2}
\bigl\||V|^{\frac78}\bigr\|_{L^{\frac{18}7\cdot\frac{4-3p}{2p-1}}}^{\frac{4-3p}{2p-1}} \Bigr)
\bigl\|\nabla|\ot|^{\frac34}\|_{L^2}^{\frac23}\\
&\leq  \Bigl(\bigl\|\nabla|\bt|^{\frac32}\bigr\|_{L^2}^{\frac23}
\bigl\|\nabla|B|^{\frac34}\bigr\|_{L^2}^{\frac23}
\bigl\||B|^{\frac34}\bigr\|_{L^2}^{\frac23}
+\|r\ut\|_{L^\infty}^{\frac{39p-37}{8(2p-1)}}
\bigl\|\nabla|V|^{\frac78}\bigr\|_{L^2}^{\frac{37-29p}{6(2p-1)}}
\bigl\||V|^{\frac78}\bigr\|_{L^2}^{\frac{23p-19}{6(2p-1)}} \Bigr)
\bigl\|\nabla|\ot|^{\frac34}\|_{L^2}^{\frac23}\\
&\leq \frac19\bigl\|\nabla |\ot|^{\frac34}\bigr\|_{L^2}^2
+\|B\|_{L^{\frac32}}^{\frac34} \bigl\|\nabla|\bt|^{\frac32}\bigr\|_{L^2}
\bigl\|\nabla|B|^{\frac34}\bigr\|_{L^2}
+\|r\ut_0\|_{L^\infty}^{\frac{3(39p-37)}{16(2p-1)}}
\bigl\|\nabla|V|^{\frac78}\bigr\|_{L^2}^{\frac{37-29p}{4(2p-1)}}
\|V\|_{L^{\frac74}}^{\frac74\cdot\frac{23p-19}{8(2p-1)}}.
\end{align*}
Substituting these two estimates into \eqref{omega321}, we obtain
\begin{align*}
\frac{d}{dt}\|\ot(t)&\|_{L^{\frac32}}^{\frac32}+
\bigl\|\nabla |\ot|^{\frac34}\bigr\|_{L^2}^2+\Bigl\|\frac{|\ot|^{\frac34}}{r}\Bigr\|_{L^2}^2
\leq C\|\eta\|_{L^p}^{\frac{p^2}{2p-1}}\|\ot\|_{L^{\frac32}}^{\frac32}
\bigl\|\nabla|\eta|^{\frac p2}\bigr\|_{L^2}^{\frac {2(2-p)}{2p-1}}\\
&+C\|B\|_{L^{\frac32}}^{\frac34} \bigl\|\nabla|\bt|^{\frac32}\bigr\|_{L^2}
\bigl\|\nabla|B|^{\frac34}\bigr\|_{L^2}
+C\|r\ut_0\|_{L^\infty}^{\frac{3(39p-37)}{16(2p-1)}}
\bigl\|\nabla|V|^{\frac78}\bigr\|_{L^2}^{\frac{37-29p}{4(2p-1)}}
\|V\|_{L^{\frac74}}^{\frac74\cdot\frac{23p-19}{8(2p-1)}}.
\end{align*}
Then a use of Gronwall's inequality, combining with the estimates \eqref{thmmain3} and \eqref{estimatebL3}, gives
\begin{align*}
&\|\ot(t)\|_{L^{\frac32}}^{\frac32}+
\int_0^t \bigl\|\nabla |\ot|^{\frac34}\bigr\|_{L^2}^2+\Bigl\|\frac{|\ot|^{\frac34}}{r}\Bigr\|_{L^2}^2\,dt'\\
&\leq C \exp\Bigl\{C\int_0^t \|\eta\|_{L^p}^{p\cdot\frac{p}{2p-1}}
\bigl\|\nabla|\eta|^{\frac p2}\bigr\|_{L^2}^{\frac {2(2-p)}{2p-1}}\, dt'\Bigr\}
 \cdot \Bigl(\|\ot_0\|_{L^{\frac32}}^{\frac32}\\
&\qquad+\int_0^t \|B\|_{L^{\frac32}}^{\frac34} \bigl\|\nabla|\bt|^{\frac32}\bigr\|_{L^2}
\bigl\|\nabla|B|^{\frac34}\bigr\|_{L^2}+ \|r\ut_0\|_{L^\infty}^{\frac{3(39p-37)}{16(2p-1)}}
\bigl\|\nabla|V|^{\frac78}\bigr\|_{L^2}^{\frac{37-29p}{4(2p-1)}}
\|V\|_{L^{\frac74}}^{\frac74\cdot\frac{23p-19}{8(2p-1)}}\,dt'\Bigr)\\
&\leq C \exp\Bigl\{C\|\eta\|_{L^\infty(0,t;L^p)}^{p\cdot\frac{p}{2p-1}}
\cdot t^{\frac{3(p-1)}{2p-1}}
\bigl(\int_0^t \bigl\|\nabla|\eta|^{\frac p2}\bigr\|_{L^2}^{2} dt'\bigr)^{\frac {2-p}{2p-1}}\Bigr\}\\
&\qquad\qquad\quad\cdot\Bigl(\|\ot_0\|_{L^{\frac32}}^{\frac32}
+\|B\|_{L^\infty\bigl(0,t;L^{\frac32}\bigr)}^{\frac34}
\bigl(\int_0^t \bigl\|\nabla|B|^{\frac34}\bigr\|_{L^2}^2\, dt'\bigr)^{\frac12}
\bigl(\int_0^t \bigl\|\nabla|\bt|^{\frac32}\bigr\|_{L^2}^2\, dt'\bigr)^{\frac12}\\
&\qquad\qquad\qquad\qquad\qquad+\|r\ut_0\|_{L^\infty}^{\frac{3(39p-37)}{16(2p-1)}}
\|V\|_{L^\infty\bigl(0,t;L^{\frac74}\bigr)}^{\frac74\cdot\frac{23p-19}{8(2p-1)}}
\cdot t^{\frac{45(p-1)}{8(2p-1)}}\bigl(\int_0^t \bigl\|\nabla|V|^{\frac78}\bigr\|_{L^2}^2\, dt'\bigr)\Bigr)^{\frac{37-29p}{8(2p-1)}}\\
&\leq C \exp\Bigl\{C(2M_0)^{\frac{2}{2p-1}}
\cdot t^{\frac{3(p-1)}{2p-1}}\Bigr\}
\cdot \Bigl(\|\ot_0\|_{L^{\frac32}}^{\frac32}+\|B_0\|_{L^{\frac32}}^{\frac32}
\|\bt_0\|_{L^{3}}^{\frac32}\exp\bigl\{C(2M_0)^{\frac{2}{2p-1}}\cdot t^{\frac{3(p-1)}{2p-1}}\bigr\}\\
&\qquad\qquad\qquad\qquad\qquad\qquad\qquad\qquad\qquad\qquad\qquad
+\|r\ut_0\|_{L^\infty}^{\frac{3(39p-37)}{16(2p-1)}}
(2M_0)^{\frac{3(3-p)}{4(2p-1)}}\cdot t^{\frac{45(p-1)}{8(2p-1)}}\Bigr).
\end{align*}
From this, and a use of the elementary inequality $x^{\frac{15}{8}} e^{Cx}\leq e^{Cx}$, we achieve
\begin{equation}\label{omega322}
\|\ot(t)\|_{L^{\frac32}}^{\frac32}+
\int_0^t \bigl\|\nabla |\ot|^{\frac34}\bigr\|_{L^2}^2+\Bigl\|\frac{|\ot|^{\frac34}}{r}\Bigr\|_{L^2}^2\,dt'
\leq C N_0 \exp\Bigl\{C(2M_0)^{\frac{2}{2p-1}}
\cdot t^{\frac{3(p-1)}{2p-1}}\Bigr\},
\end{equation}
here we denote $N_0\eqdefa \|\omega_0\|_{L^{\frac32}}^{\frac32}+\|J_0\|_{L^{\frac32}}^{3}
+\|r\ut_0\|_{L^\infty}^{\frac{3(39p-37)}{16(2p-1)}}
(2M_0)^{-\frac{3(p+2)}{4(2p-1)}}$.

To estimate $\widetilde{\omega}=(\orr,\oz)$,
applying the $L^{\frac32}$ estimate to the first and third equations of \eqref{axiomega}
respectively, then putting these two estimates together, we obtain
\begin{equation}\label{oroz1}\begin{split}
\frac23\frac{d}{dt}\bigl(\|\orr(t)\|_{L^{\frac32}}^{\frac32}+
\|\oz&(t)\|_{L^{\frac32}}^{\frac32}\bigr)
+\frac89\bigl\|\nabla |\orr|^{\frac34}\bigr\|_{L^2}^2+\frac89\bigl\|\nabla |\oz|^{\frac34}\bigr\|_{L^2}^2
+\Bigl\|\frac{|\orr|^{\frac34}}{r}\Bigr\|_{L^2}^2\\
&\leq \int_{\R^3} |(\orr\pa_r+\oz\pa_z)\ur| \cdot|\orr|^{\frac12}
+|(\orr\pa_r+\oz\pa_z)\uz| \cdot|\oz|^{\frac12}\,dx.
\end{split}\end{equation}
In view of \eqref{divcurl} and the Biot-Savart law \eqref{Biot},
for $\widetilde{u}=(\ur e_r+\uz e_z)$, we have
\begin{equation}
\|\nabla\widetilde{u}\|_{L^m}\leq C \frac{m^2}{m-1}\|\ot\|_{L^m},\quad\forall 1<m<\infty.
\end{equation}
And then we can estimate the right hand side of \eqref{oroz1} as follows
\begin{align*}
\int_{\R^3} |(\orr\pa_r& +\oz\pa_z)\ur| \cdot|\orr|^{\frac12}
+|(\orr\pa_r+\oz\pa_z)\uz| \cdot|\oz|^{\frac12}\,dx\\
&\leq \bigl(\|\orr\|_{L^{\frac32}}+\|\oz\|_{L^{\frac32}}\bigr)
\|\nabla\widetilde{u}\|_{L^{\frac92}}
\Bigl(\bigl\||\orr|^{\frac12}\bigr\|_{L^9}+\bigl\||\oz|^{\frac12}\bigr\|_{L^9}\Bigr)\\
&\leq \bigl(\|\orr\|_{L^{\frac32}}+\|\oz\|_{L^{\frac32}}\bigr)
\bigl\|\nabla|\ot|^{\frac34}\bigr\|_{L^2}^{\frac43}
\Bigl(\bigl\|\nabla|\orr|^{\frac34}\bigr\|_{L^2}^{\frac23}
+\bigl\|\nabla|\oz|^{\frac34}\bigr\|_{L^2}^{\frac23}\Bigr)\\
&\leq \frac19\Bigl(\bigl\|\nabla|\orr|^{\frac34}\bigr\|_{L^2}^{2}
+\bigl\|\nabla|\oz|^{\frac34}\bigr\|_{L^2}^{2}\Bigr)
+\bigl(\|\orr\|^{\frac32}_{L^{\frac32}}+\|\oz\|^{\frac32}_{L^{\frac32}}\bigr)
\bigl\|\nabla|\ot|^{\frac34}\bigr\|_{L^2}^2.
\end{align*}
Substituting this estimate into \eqref{oroz1}, absorbing the term $\frac19\bigl(\bigl\|\nabla|\orr|^{\frac34}\bigr\|_{L^2}^{2}
+\bigl\|\nabla|\oz|^{\frac34}\bigr\|_{L^2}^{2}\bigr)$ on the right, then a use of Gronwall's inequality
and the estimate \eqref{omega322}, we achieve
\begin{equation}\label{omega323}\begin{split}
\|\orr(t)\|_{L^{\frac32}}^{\frac32}+&
\|\oz(t)\|_{L^{\frac32}}^{\frac32}
+\int_0^t \bigl\|\nabla |\orr|^{\frac34}\bigr\|_{L^2}^2+\bigl\|\nabla |\oz|^{\frac34}\bigr\|_{L^2}^2
+\Bigl\|\frac{|\orr|^{\frac34}}{r}\Bigr\|_{L^2}^2\, dt'\\
&\leq C\bigl(\|\orr_0\|_{L^{\frac32}}^{\frac32}+
\|\oz_0\|_{L^{\frac32}}^{\frac32}\bigr)\cdot
\exp\Bigl\{C N_0 \exp\bigl\{C(2M_0)^{\frac{2}{2p-1}}
\cdot t^{\frac{3(p-1)}{2p-1}}\bigr\}\Bigr\}\\
&\leq C \exp\Bigl\{C N_0 \exp\bigl\{C(2M_0)^{\frac{2}{2p-1}}
\cdot t^{\frac{3(p-1)}{2p-1}}\bigr\}\Bigr\}.
\end{split}\end{equation}
Now combining the estimates \eqref{omega322} and \eqref{omega323}, and the point-wise estimate
\begin{align*}
\nabla|\omega|^{\frac34}\thicksim |\omega|^{-\frac14}|\nabla\omega|
&\thicksim|\omega|^{-\frac14}
\Bigl(|\nabla\orr|+|\nabla\ot|+|\nabla\oz|+\bigl|\frac{\orr}{r}\bigr|+\bigl|\frac{\ot}{r}\bigr|\Bigr)\\
&\lesssim \bigl|\nabla|\orr|^{\frac34}\bigr|+\bigl|\nabla|\ot|^{\frac34}\bigr|+\bigl|\nabla|\oz|^{\frac34}\bigr|
+\frac{|\orr|^{\frac34}}{r}+\frac{|\ot|^{\frac34}}{r},\quad\mbox{a.e.}
\end{align*}
we achieve
\begin{equation}\label{omega32}
\|\omega(t)\|_{L^{\frac32}}^{\frac32}
+\int_0^t \bigl\|\nabla |\omega|^{\frac34}\bigr\|_{L^2}^2\, dt'
\leq C
\exp\Bigl\{C N_0 \exp\bigl\{C(2M_0)^{\frac{2}{2p-1}}
\cdot t^{\frac{3(p-1)}{2p-1}}\bigr\}\Bigr\}.
\end{equation}

$\bullet$ \underline{The $L^{\frac32}$ estimate of $J$}

Noting that $b=\bt e_\theta$, we can write $J=\Jr e_r+\Jz e_z$ with
$J^r=-\pa_z b^\theta,~ J^z=\pa_r b^\theta+\frac{b^\theta}{r}$.
Using the equations for $\bt$ given in \eqref{axi}, and the divergence-free conditions
$\pa_r (r u^r)+\pa_z (r u^z)=0,~\pa_r (r J^r)+\pa_z (r J^z)=0$,
it is not difficult to derive the equations for $J$ as follows:
\begin{equation}\label{axiJ}
\left\{
\begin{array}{l}
\displaystyle \pa_t \Jr+(u^r\pa_r+u^z\pa_z) \Jr-(\Delta-\frac{1}{r^2})\Jr=
(\Jr\pa_r+\Jz\pa_z)\ur+2\frac{\ur}{r}\Jr-2\frac{\bt}{r}\pa_z \ur,\\
\displaystyle \pa_t J^z+(u^r\pa_r+u^z\pa_z) J^z-\Delta J^z=
(\Jr\pa_r+\Jz\pa_z)\uz+2\frac{\ur}{r}\Jz+2\frac{\bt}{r}\bigl(\pa_r \ur-\frac{\ur}{r}\bigr).
\end{array}
\right.
\end{equation}
Applying the $L^{\frac32}$ estimate of \eqref{axiJ},
and in view of the following point-wise estimate
\begin{equation}\label{Biotutilde}
\bigl|{\nabla}\ur\bigr|+\bigl|{\nabla}\uz\bigr|+\bigl|\frac{\ur}{r}\bigr|
\leq C|\nabla (\ur e_r+\uz e_z)|,\quad\mbox{a.e.}
\end{equation}
as well as the Biot-Savart law \eqref{Biot}, we obtain
\begin{align*}
&\frac{d}{dt}\bigl(\|\Jr(t)\|_{L^{\frac32}}^{\frac32}+
\|\Jz(t)\|_{L^{\frac32}}^{\frac32}\bigr)
+\bigl\|\nabla |\Jr|^{\frac34}\bigr\|_{L^2}^2+\bigl\|\nabla |\Jz|^{\frac34}\bigr\|_{L^2}^2
+\Bigl\|\frac{|\Jr|^{\frac34}}{r}\Bigr\|_{L^2}^2\\
&\leq C \int_{\R^3} \bigl(|\Jr|^{\frac32}+|\Jz|^{\frac32}\bigr)|\nabla\widetilde{u}|
+|B|\bigl(|\Jr|^{\frac12}+|\Jz|^{\frac12}\bigr)|\nabla\widetilde{u}|\,dx\\
&\leq C \bigl(\|\Jr\|_{L^{\frac52}}^{\frac32}+\|\Jz\|_{L^{\frac52}}^{\frac32}\bigr)\|\ot\|_{L^{\frac52}}
+C\|B\|_{L^{\frac52}}\bigl(\|\Jr\|_{L^{\frac52}}^{\frac12}+\|\Jz\|_{L^{\frac52}}^{\frac12}\bigr)\|\ot\|_{L^{\frac52}}\\
&\leq C \Bigl( \|\Jr\|_{L^{\frac32}}^{\frac35} \bigl\|\nabla |\Jr|^{\frac34}\bigr\|_{L^2}^{\frac65}
+ \|\Jz\|_{L^{\frac32}}^{\frac35} \bigl\|\nabla |\Jz|^{\frac34}\bigr\|_{L^2}^{\frac65}
+ \|B\|_{L^{\frac32}}^{\frac35} \bigl\|\nabla |B|^{\frac34}\bigr\|_{L^2}^{\frac65}\Bigr)
 \|\ot\|_{L^{\frac32}}^{\frac25} \bigl\|\nabla |\ot|^{\frac34}\bigr\|_{L^2}^{\frac45}\\
&\leq \frac12\bigl\|\nabla |\Jr|^{\frac34}\bigr\|_{L^2}^2+\frac12\bigl\|\nabla |\Jz|^{\frac34}\bigr\|_{L^2}^2
+C\bigl(\|\Jr\|_{L^{\frac32}}^{\frac32}+\|\Jz\|_{L^{\frac32}}^{\frac32}\bigr)
 \|\ot\|_{L^{\frac32}} \bigl\|\nabla |\ot|^{\frac34}\bigr\|_{L^2}^{2}\\
&\qquad\qquad\qquad\qquad\qquad\qquad\qquad\qquad\qquad\qquad\qquad
+\|B\|_{L^{\frac32}}^{\frac35} \bigl\|\nabla |B|^{\frac34}\bigr\|_{L^2}^{\frac65}
 \|\ot\|_{L^{\frac32}}^{\frac25} \bigl\|\nabla |\ot|^{\frac34}\bigr\|_{L^2}^{\frac45}.
\end{align*}
Absorbing the term $\frac12\bigl\|\nabla |\Jr|^{\frac34}\bigr\|_{L^2}^2
+\frac12\bigl\|\nabla |\Jz|^{\frac34}\bigr\|_{L^2}^2$
on the right, then a use of Gronwall's inequality
and the estimates \eqref{estimateB},~\eqref{omega322}, leads to
\begin{align*}
&\|\Jr(t)\|_{L^{\frac32}}^{\frac32}+
\|\Jz(t)\|_{L^{\frac32}}^{\frac32}+\int_0^t
\bigl\|\nabla |\Jr|^{\frac34}\bigr\|_{L^2}^2+\bigl\|\nabla |\Jz|^{\frac34}\bigr\|_{L^2}^2
+\Bigl\|\frac{|\Jr|^{\frac34}}{r}\Bigr\|_{L^2}^2\, dt'\\
&\lesssim\Bigl(\|J_0\|_{L^{\frac32}}^{\frac32}+\int_0^t
\|B\|_{L^{\frac32}}^{\frac35} \bigl\|\nabla |B|^{\frac34}\bigr\|_{L^2}^{\frac65}
 \|\ot\|_{L^{\frac32}}^{\frac25} \bigl\|\nabla |\ot|^{\frac34}\bigr\|_{L^2}^{\frac45}\, dt'\Bigr)
\cdot \exp\Bigl\{\int_0^t\|\ot\|_{L^{\frac32}} \bigl\|\nabla |\ot|^{\frac34}\bigr\|_{L^2}^{2}\, dt'\Bigr\}\\
&\lesssim\Bigl(\|J_0\|_{L^{\frac32}}^{\frac32}
+\|B_0\|_{L^{\frac32}}^{\frac32} N_0^{\frac23} \exp\bigl\{C(2M_0)^{\frac{2}{2p-1}}
\cdot t^{\frac{3(p-1)}{2p-1}}\bigr\}\Bigr)
\cdot\exp\Bigl\{C N_0^{\frac53} \exp\bigl\{C(2M_0)^{\frac{2}{2p-1}}
\cdot t^{\frac{3(p-1)}{2p-1}}\bigr\}\Bigr\}\\
&\lesssim\Bigl(N_0^{\frac12}
+ N_0^{\frac76} \exp\bigl\{C(2M_0)^{\frac{2}{2p-1}}
\cdot t^{\frac{3(p-1)}{2p-1}}\bigr\}\Bigr)
\cdot\exp\Bigl\{C N_0^{\frac53} \exp\bigl\{C(2M_0)^{\frac{2}{2p-1}}
\cdot t^{\frac{3(p-1)}{2p-1}}\bigr\}\Bigr\}\\
&\lesssim \exp\Bigl\{C N_0^{\frac53} \exp\bigl\{C(2M_0)^{\frac{2}{2p-1}}
\cdot t^{\frac{3(p-1)}{2p-1}}\bigr\}\Bigr\}.
\end{align*}
Combining this estimate, and the point-wise estimate
\begin{align*}
\nabla|J|^{\frac34}\thicksim |J|^{-\frac14}|\nabla J|
&\thicksim|J|^{-\frac14}
\Bigl(|\nabla\Jr|+|\nabla\Jz|+\bigl|\frac{\Jr}{r}\bigr|\Bigr)\\
&\lesssim \bigl|\nabla|\Jr|^{\frac34}\bigr|+\bigl|\nabla|\Jz|^{\frac34}\bigr|
+\frac{|\Jr|^{\frac34}}{r},\quad\mbox{a.e.}
\end{align*}
we achieve
\begin{equation}\label{estimateJ}
\|J(t)\|_{L^{\frac32}}^{\frac32}+\int_0^t
\bigl\|\nabla |J|^{\frac34}\bigr\|_{L^2}^2\, dt'
\leq C\exp\Bigl\{C N_0^{\frac53} \exp\bigl\{C(2M_0)^{\frac{2}{2p-1}}
\cdot t^{\frac{3(p-1)}{2p-1}}\bigr\}\Bigr\}.
\end{equation}

\section{The $L^1(0,t;\Lip)$ estimate of $u$}\label{SecL1}

Let $\cP\eqdefa \rm{Id}+\nabla(-\Delta)^{-1}\div$ denote the Leray projector onto divergence-free vector fields.
For any selected $j\in\N$, applying $\dj\cP$ to the velocity equation in \eqref{MHD} gives
\begin{equation*}
\d_t(\dj u)-\Delta(\dj u)+\cP\dj\div (u\otimes u)=\cP\dj \div (b\otimes b),
\end{equation*}
where we have used the divergence-free condition on $u$ and $b$. Then Duhamel formula gives
\begin{equation}\label{mildsol}
\dj u(t)=e^{t\Delta}(\dj u_0)-\int_0^t e^{(t-\tau)\Delta}\,\cP\dj\div (u\otimes u-b\otimes b)(\tau,x)\,d\tau.
\end{equation}
Noting that $\cP$ satisfies the condition \eqref{conditionMihlin} with $m=0$,
thus we can use Lemma \ref{GBern} and Lemma \ref{heatflow} to get that,
there holds uniformly for every $j\in\Z$:
\begin{align*}
\|\dj u\|_{L_t^1 L^\infty}
&\lesssim\int_0^t e^{-c\tau\cdot2^{2j}}\|\dj u_0\|_{L^\infty}d\tau
+\int_0^t\bigl(\|\dj\div (u\otimes u)(\tau',\cdot)\|_{L^\infty}\ast_{\tau'}e^{-c\tau'\cdot2^{2j}}\bigr)(\tau)\,d\tau\\
&\quad\qquad+\int_0^t\bigl(\|\dj\div (b\otimes b)(\tau',\cdot)\|_{L^\infty}\ast_{\tau'}e^{-c\tau'\cdot2^{2j}}\bigr)(\tau)\,d\tau\\
&\lesssim \bigl\|e^{-c\tau\cdot2^{2j}}\bigr\|_{L^1(0,t)}\cdot\bigl(
\|\dj u_0\|_{L^\infty}+\|\dj\div (u\otimes u)\|_{L_t^1 L^\infty}
+\|\dj\div (b\otimes b)\|_{L_t^1 L^\infty}\bigr)\\
&\lesssim2^{-2j}\cdot\bigl(
\|\dj u_0\|_{L^\infty}+\|\dj\div (u\otimes u)\|_{L_t^1 L^\infty}
+\|\dj\div (b\otimes b)\|_{L_t^1 L^\infty}\bigr)
\end{align*}
Multiplying both sides by $2^j$
and then summing up over all $j\in\Z$, we obtain
\begin{equation}\begin{split}\label{Lip1}
\|u\|_{L_t^1B^{1}_{\infty,1}}
&\lesssim \sum_{j\in\Z}2^{-j}\cdot\bigl(
\|\dj u_0\|_{L^\infty}+\|\dj\div (u\otimes u)\|_{L_t^1 L^\infty}
+\|\dj\div (b\otimes b)\|_{L_t^1 L^\infty}\bigr)\\
&\lesssim \|u_0\|_{B^{-1}_{\infty,1}}
+\|u\otimes u\|_{L_t^1 \B0}+\|b\otimes b\|_{L_t^1 \B0}.
\end{split}\end{equation}
Exactly along the same line, and a use of Minkowski's inequality, we obtain
\begin{equation}\begin{split}\label{Lip2}
\|u\|_{L_t^\infty B^{-1}_{\infty,1}}
&\leq \sum_{j\in\Z}2^{-j} \|\dj u\|_{L_t^\infty L^\infty}\\
&\lesssim \sum_{j\in\Z}2^{-j} \bigl(
\|\dj u_0\|_{L^\infty}+\|\dj\div (u\otimes u)\|_{L_t^1 L^\infty}
+\|\dj\div (b\otimes b)\|_{L_t^1 L^\infty}\bigr)\\
&\lesssim\|u_0\|_{B^{-1}_{\infty,1}}
+\|u\otimes u\|_{L_t^1 \B0}+\|b\otimes b\|_{L_t^1 \B0}.
\end{split}\end{equation}

Using Lemma \ref{productlaw}, and the embedding $B^{\frac12}_{6,1}\hookrightarrow
\B0\hookrightarrow L^\infty$, we have
\begin{equation}\label{utimesu}
\|u\otimes u\|_{L_t^1 \B0}\leq C\|u\otimes u\|_{L_t^1 B^{\frac12}_{6,1}}
\leq C\|u\|_{L_t^2 B^{\frac12}_{6,1}}^2.
\end{equation}
By Lemma \ref{Bern}, we get, for any integer $N$, that
\begin{equation}\begin{split}\label{N}
\|u\|_{B^{\frac12}_{6,1}}&=\sum_{j\in\Z}2^{\frac12j}\|\dj u\|_{L^6}\\
&\lesssim\sum_{j\leq N} 2^j\|\dj u\|_{L^3}+\sum_{j> N} 2^{-\frac{j}3}\|\nabla\dj u\|_{L^{\frac92}}\\
&\lesssim 2^N\|\omega\|_{L^{\frac32}}+2^{-\frac{N}{3}}\bigl\|\nabla|\omega|^{\frac34}\bigr\|_{L^2}^{\frac43},
\end{split}\end{equation}
where we have used the Biot-Savart law \eqref{Biot} and the Sobolev embedding such that
$$\|\nabla u\|_{L^{\frac92}}\lesssim\|\omega\|_{L^{\frac92}}\lesssim
\bigl\|\nabla|\omega|^{\frac34}\bigr\|_{L^2}^{\frac43}.$$
Choosing a proper $N$ satisfying $2^N\thicksim\bigl\|\nabla|\omega|^{\frac34}\bigr\|_{L^2}
\|\omega\|_{L^{\frac32}}^{-\frac34}$ in \eqref{N}, leads to
$$\|u\|_{B^{\frac12}_{6,1}}\lesssim\bigl\|\nabla|\omega|^{\frac34}\bigr\|_{L^2}
\|\omega\|_{L^{\frac32}}^{\frac14}.$$
Substituting this estimate into \eqref{utimesu}, and using \eqref{omega32}, we obtain
\begin{equation}\begin{split}\label{estimateutimesu}
\|u\otimes u\|_{L_t^1 \B0}& \leq C\bigl\|\nabla|\omega|^{\frac34}\bigr\|_{L^2(0,t;L^2)}^2
\|\omega\|_{L^\infty(0,t;L^{\frac32})}^{\frac12}\\
&\leq C
\exp\Bigl\{C N_0 \exp\bigl\{C(2M_0)^{\frac{2}{2p-1}}
\cdot t^{\frac{3(p-1)}{2p-1}}\bigr\}\Bigr\}.
\end{split}\end{equation}

Exactly along the same line, by using \eqref{estimateJ}, we can obtain
\begin{equation}\label{estimatebtimesb}
\|b\otimes b\|_{L_t^1 \B0}\leq C
\exp\Bigl\{C N_0^{\frac53} \exp\bigl\{C(2M_0)^{\frac{2}{2p-1}}
\cdot t^{\frac{3(p-1)}{2p-1}}\bigr\}\Bigr\}.
\end{equation}

Substituting the estimates \eqref{estimateutimesu},~\eqref{estimatebtimesb} into \eqref{Lip1} and \eqref{Lip2},
finally we achieve
\begin{equation}
\|u\|_{L_t^\infty B^{-1}_{\infty,1}}+\|u\|_{L_t^1 B^{1}_{\infty,1}}
\leq \|u_0\|_{B^{-1}_{\infty,1}}
+C\exp\Bigl\{C \bigl(N_0+N_0^{\frac53}\bigr) \exp\bigl\{C(2M_0)^{\frac{2}{2p-1}}
\cdot t^{\frac{3(p-1)}{2p-1}}\bigr\}\Bigr\},
\end{equation}
which is the desired $L^1(0,t;\Lip)$ estimate of $u$. This completes the proof of the theorem.

\smallskip
\noindent {\bf Acknowledgments.}
This work was
done when I was visiting Morningside Center of Mathematics, Chinese Academy of Sciences. I appreciate the
hospitality of MCM.


\medskip


\begin{thebibliography}{9999}

 \bibitem{Abidi} H. Abidi,
{R\'{e}sultats de r\'{e}gularit\'{e} de solutions axisym\'{e}triques pour le syst\`{e}me de Navier-Stokes},
{\it Bull. Sci. Math.},
   {\bf 132} (2008), no. 7,
 pages 592每624.

 \bibitem{Paicu} H. Abidi and M. Paicu,
{Global existence for the magnetohydrodynamic system in critical spaces},
{\it Proc. Roy. Soc. Edinburgh Sect. A},
   {\bf 138} (2008), no.3,
 pages 447每476.

\bibitem {BCD} H. Bahouri, J.-Y. Chemin and R. Danchin, {\it Fourier Analysis and
Nonlinear Partial Differential Equations}, Grundlehren der
Mathematischen Wissenschaften, Springer, 2010.

\bibitem{CKN} L. Caffarelli, R. Kohn and L. Nirenberg,
{Partial regularity of suitable weak solutions of the Navier-Stokes equations},
{\it Comm. Pure Appl. Math.},
   {\bf 35} (1982), No. 6,
 pages 771每831.


\bibitem{CL02} D. Chae and J. Lee,  On the regularity of the axisymmetric solutions of the Navier-Stokes equations,
{\it  Math. Z.}, {\bf  239 }  (2002),   645-671.

\bibitem{Chenhui} H. Chen, D. Fang and T. Zhang,
{Regularity of 3D axisymmetric Navier-Stokes equations},
{\it Discrete Contin. Dyn. Syst.},
   {\bf 37} (2017), No. 4,
 pages 1923每1939.

 \bibitem{Lions} G. Duvaut and J. L. Lions,
{In\'{e}quations en thermo\'{e}lasticit\'{e} et magn\'{e}tohydrodynamique},
{\it Arch. Ration. Mech. Anal.},
   {\bf 46} (1972),
 pages 241-279.

  \bibitem{La} O.~ A.  Lady$\breve{z}$enskaja, Unique global solvability of the three-dimensional Cauchy problem for the Navier-Stokes
 equations in the presence of axial symmetry, (Russian) {\it Zap. Nau$\breve{c}$n. Sem. Leningrad. Otdel. Mat. Inst. Steklov. (LOMI)},
 {\bf 7} (1968), 155-177.

\bibitem{Lei}
Z. Lei,
     {On axially symmetric incompressible magnetohydrodynamics in three dimensions},
   {\it J. Differential Equations},
   {\bf 259} (2015),
 pages 3202-3215.

 \bibitem{LeiZhang} Z. Lei and Q. Zhang,
Criticality of the axially symmetric Navier-Stokes Equations,
 arXiv: 1505.02628v2

\bibitem{LMNP} S. Leonardi, J. M\'{a}lek, J.  Ne$\breve{}{c}$as and M. Pokorny,  On axially
symmetric flows in $\mathbb{R}^3,$ {\it  Z. Anal. Anwendungen}, {\bf
18} (1999),  639-649.

\bibitem{LZ} Y. Liu and P. Zhang,
 On the global well-posedness of 3-D axi-symmetric Navier-Stokes system with small swirl component,
 arXiv: 1702.06279v1

\bibitem{Temam}
M. Sermange and R. Temam,
     {Some mathematical questions related to the MHD equations},
   {\it  Comm. Pure Appl. Math.},
   {\bf  36} (1983), no. 5,
 pages 635每664.

\bibitem{UY} M.~ R. Ukhovskii, and V. I. Iudovich,  Axially symmetric flows of ideal and viscous fluids
filling the whole space, {\it J. Appl. Math. Mech.}, {\bf 32} (1968)
52-61.

\bibitem{Wei}
D. Wei,
     {Regularity criterion to the axially symmetric Navier-Stokes equations},
   {\it  J. Math. Anal. Appl.},
   {\bf 435} (2016), no. 1,
 pages 402每413.

\bibitem{ZZT2} P. Zhang and T. Zhang, Global axi-symmetric solutions to 3-D   Navier-Stokes
System,  {\it Int. Math. Res. Not. IMRN }, Vol. 2013, No. 3,
610-642.


\end{thebibliography}
\end{document}